%% file: RCE-fields.tex
\documentclass[12pt]{article}

\usepackage{amssymb,amsmath,amsthm}

\usepackage{hyperref}

\usepackage[height=2em,width=2em,PS]{diagrams}

\input{macros}

\newtheorem{prop}{Proposition}
\newtheorem{cor}[prop]{Corollary}
\newtheorem{theorem}[prop]{Theorem}
\newtheorem{lemma}[prop]{Lemma}

\newtheorem{conjecture}[prop]{Conjecture}

\newtheorem{claim}[prop]{Claim}

\theoremstyle{definition}
\newtheorem{defn}[prop]{Definition}
\newtheorem{example}[prop]{Example}

\newcommand{\mK}{\mathcal{K}}

\newcommand{\ba}{\bar{a}}

\newcommand{\z}{\bar{z}}
\newcommand{\y}{\bar{y}}
\newcommand{\bt}{\bar{t}}
\newcommand{\bs}{\bar{s}}
\newcommand{\br}{\bar{r}}
\newcommand{\bu}{\bar{u}}
\newcommand{\bv}{\bar{v}}

\newcommand{\bq}{\bar{q}}
\newcommand{\be}{\bar{e}}
\newcommand{\bl}{\bar{l}}

\DeclareMathOperator{\img}{img}
\DeclareMathOperator{\scl}{scl}

\title{Real Closed Exponential Subfields of Pseudoexponential Fields}

\author{Ahuva C. Shkop \footnote{The author was generously supported by the Center for Advanced Studies in Mathematics at Ben-Gurion University of the Negev.}}

\date{}

\begin{document}

\maketitle

\begin{abstract}
In this paper, we prove that a pseudoexponential field has continuum many non-isomorphic countable real closed exponential subfields, each with an order preserving exponential map which is surjective onto the nonnegative elements.  Indeed, this is true of any algebraically closed exponential field satisfying Schanuel's conjecture.
\end{abstract}

\section{Introduction}

For many decades, the first order theory of complex exponentiation, i.e. the theory of $\C_{exp}:= \langle \C, +,\cdot,0,1,e^z \rangle$ has been very difficult to study and many questions stemming from model theory, geometry, and number theory remain open.  One of the most famous of these problems is the following conjecture from the 1960's due to Schanuel:

\begin{conjecture}\text{(Schanuel's Conjecture)} If $\{z_1,...,z_n \}\subset \C$, then 

$\td_\Q(z_1,...,z_n,e^{z_1},...,e^{z_n})$, where $\td_\Q$ is the transcendence
degree over $\Q$, is at least the
 $\Q$ linear dimension of $\{z_1,...,z_n\}.$

\end{conjecture}

In 2001, Zilber combined this and many other open questions into one intriguing conjecture.  In \cite{P-exp}, Zilber constructs a class of exponential fields known as pseudoexponential fields.  A pseudoexponential field, $K$, satisfies the following six properties:

\begin{enumerate}
\item $K$ is an algebraically closed field of characteristic zero.
\item $\exp$ is a surjective homomorphism from the additive group of $K$ onto the multiplicative group of $K$.
\item There is some transcendental $\nu$ so that $\ker(\exp_K) = \nu\Z$.
\item Schanuel property: If $a_1,\dots,a_n \in K$ are $\Q$-linearly independent, then $td_\Q(a_1,\dots,a_n, \exp(a_1),...,\exp(a_n))\geq n$. (note: This is equivalent to Schanuel's conjecture for $K$)
\item Exponential Closure:  We need the following definitions to state this property but we will not refer to them for the remainder of the paper.  Let $\alpha \in \N$ and  $G_\alpha(K):= K^\alpha \times (K^*)\alpha$. For $[C]=(c_{i,j})$ an $r \times \alpha$ matrix of integers, let $[C]:G_\alpha(K) \to G_r(K)$ be the function which acts additively on the first $\alpha$ coordinates and multiplicatively on the last $\alpha$ coordinates, i.e. $[C](\z,\y)= (u_1,...,u_r,v_1,...,v_r)$ where $$u_i = \sum_{j=1}^\alpha c_{i,j}z_j  \text{  and  }  v_i =
\prod_{j=1}^\alpha y_j^{c_{i,j}}.$$   An irreducible Zariski closed $V \subseteq K^\alpha \times (K^*)^\alpha $ is \emph{rotund} if $\dim([C](V))\geq r$ for any $r \times \alpha$ matrix of integers $C$ of rank $r$ where $1\leq r\leq \alpha$.  We say $V$ is \emph{free} if it is not contained in a closed set given by equations of the form $$\{(\bu,\bv): \prod_{i=1}^\alpha v_i^{m_i}=b\}$$ or
$$\{(\bu,\bv):\sum_{i=1}^\alpha m_iu_i=b\}$$ for any $m_1,...,m_\alpha \in \Z$ and $b \in K$.

\medskip

Given these definitions, the exponential closure property can be stated as follows:

\medskip

If $V\subseteq K^\alpha \times (K^*)^\alpha $ is irreducible, rotund, and free, then for any finite $A \in K$ there is $(a_1,\dots,a_\alpha, \exp(a_1),\dots,\exp(a_\alpha))\in V$ a generic point in $V$ over $A$.

\item Countable Closure: We will state this property in terms of the schanuel predimension $\delta$.  For finite $X \subset K$, let $$\delta(X):= td_\Q(X,\exp(X))- \Q \text{-l.d.}(X)$$

    \noindent where $\Q$-l.d.$(X)$ is the $\Q$-linear dimension of the span of $X$. $\delta$ is a predimension.  Notice that the Schanuel property implies that $\delta(X) \geq 0$.  Therefore, the following is always defined:
    $$d(X) = \min\{\delta (Y): \text{ Y is finite and } X \subseteq Y \subseteq K \}$$

    \noindent We can now define the Schanuel closure of any set $S\subseteq K$: $$\scl(S) = \{y \in K : \exists X \subseteq_{fin} S, \delta(Xy)=\delta(X)\}.$$

    \medskip

    \noindent Then countable closure states that the Schanuel closure of a finite set is countable.

 \end{enumerate}

\noindent \emph{Note: Schanuel closure gives a pregeometry on K.}(For the definition of pregeometry, see \cite{Dave's Book})
\medskip

\noindent  These axioms classify pseudoexponential fields.  In \cite{P-exp}, Zilber proved the following:

\begin{theorem}\text(Zilber) For $\kappa$ uncountable, there is a unique pseudoexponential field of size $\kappa$ and it has $2^k$ isomorphisms. Furthermore, pseudoexponential fields are quasiminimal, i.e. every definable subset of a pseudoexponential field is countable or co-countable.
\end{theorem}

This leads to the following question:  Is $\C_{exp}$ the unique pseudoexponential field of size continuum? Zilber conjectured that $\C_{exp}$ is indeed the pseudoexponential field of size $2^{\aleph_0}$.  It is clear that $\C_{exp}$ satisfies properties 1,2, and 3. In \cite{P-exp}  Zilber proved that $\C_{exp}$ satisfies countable closure. This paper explores a fundamental consequence of Zilber's conjecture.

From this point on, let $\mK$ be a fixed pseudoexponential field of size $\kappa$. If $\kappa = 2^{\aleph_0}$ and  $\mK$ is isomorphic to $\C_{\exp}$, then $\mK$ contains an exponential subfield isomorphic to $\R_{\exp}$.  Motivated by this observation, we will prove the following theorem:

\begin{theorem}
Let K be any algebraically closed exponential field satisfying Schanuel's conjecture (such as the pseudoexponential field $\mK$). Then there are continuum many non-isomorphic (as fields) countable real closed exponential subfields of $K$, each with an order-preserving exponential map which is surjective onto the non-negative elements.
\end{theorem}

We prove this theorem in two steps, first constructing real closed exponential fields where the exponential map is not surjective and then showing how to construct them so that every positive element is in the image of the exponential map.  It is easier to see how this construction works in two steps, rather than one, and the results of the first construction are more examples of real closed exponential subfields of $\mK$.

\noindent Throughout this paper, we make use of the following notation and convention:

\medskip

- We use the tuple notation to denote a finite subset. i.e. $\bt \subset T$ is some finite set $t_1,...,t_n$ in $T$.

- For any set $A$, we  write $\langle A \rangle  $ for the $\Q$- additively linear span of $A$.

- For any set $A$, we  write $[A]$ to mean the subring of $\mK$ generated by $A$

- For any integral domain $R$,  $R^{alg}$ is the field theoretic algebraic closure of $R$.  Throughout this paper, the term "algebraic" refers to the field-theoretic notion.

- For a set $A$, $\Q$-l.d.$(A)$ is the $\Q$-linear dimension of $\langle A \rangle$.

- For any set $A$, we write $\exp(A)$ for the set $\{\exp(a): a\in A\}$.

- $\Q^{rc}$ is the real closure of the rational numbers, or equivalently, the real algebraic numbers.

- For $R$ an ordered ring, we write $R^{>0}$ for $\{r\in R : r>0\}$.

- We say $b_1,...,b_n$ are $\Q$-linearly dependent over $X$ if $\exists q_1,...,q_n \in \Q$, not all zero, such that $q_1b_1+\cdots +q_nb_n \in \langle X\rangle$.  We say  $b_1,...,b_n$ are $\Q$-multiplicatively dependent over $X$ if $\exists q_1,...,q_n \in \Q$, not all zero, such that $b_1^{q_1}\cdots b_n^{q_n}$ is in the multiplicative span of $X$. Unless we specify that we are referring to a multiplicative linear space, the word linear will mean additively linear.

- For a finite set $\bs$, we write $\td(\bs)$ to mean $\td(\Q(\bs)/\Q)$.

\medskip

We also make use of the following elementary facts about exponential functions:

\begin{itemize}
\item If $b \in \langle X\rangle  $ then $\exp(b)$ is algebraic over $\exp(X)$.
\item Suppose $b_1,...,b_n$ are $\Q$-linearly dependent over $X$.  Then $\exp(b_1,...,b_n)$ is $\Q$- multiplicatively dependent over $\exp(X)$.
\end{itemize}

\bigskip

\bigskip

\noindent I began this work as part of my thesis and I would like to thank my thesis advisor, David Marker, for all of his encouragement and support.

\section{Free Extensions and Formally Real Fields}

We begin with the following definitions.

\begin{defn}
In this paper, a \emph{(total) E-ring} is a $\Q$-algebra $R$ with
  no zero divisors, together with a homomorphism
  $\exp: \langle R,+ \rangle \to \langle R^*,\cdot \rangle$.

  A \emph{partial E-ring} is a $\Q$-algebra $R$ with no zero
  divisors, together with a $\Q$-linear subspace $A(R)$ of $R$ and a
  homomorphism $\exp_R: \langle A(R),+ \rangle \to \langle R^*,\cdot \rangle$. $A(R)$ is then the domain of $\exp_R$.

  An  \emph{E-field} is an E-ring which is a field.

  We say \emph{$S$ is a partial E-ring extension of $R$} if $R$ and $S$ are partial E-ring, $R\subseteq S$, and for all $r\in A(R)$,
  $\exp_S(r)=\exp_R(r)$.

\smallskip

\noindent When there is no ambiguity, we drop the subscript.
\end{defn}

The following example is an important subtlety with regards to the definition of partial E-ring extension.

\begin{example}Let $S$ be a partial E-ring.  If one considers $R = S$ and $A(R)\subsetneq A(S)$ a $\Q$-subspace of $A(S)$, then $S$ is a (proper) partial E-ring extension of $R$.
\end{example}

\begin{defn} Let $R$ be a partial E-ring.  We say $R' \supseteq R$ is a \emph{free partial E-ring extension of $R$} if
\begin{itemize}
\item $R'$ is a partial E-ring extension of $R$.
\item The domain of $\exp_{R'}$ contains $R$.
\item If $\{a_1,...,a_n\} \subset R$ is $\Q$-linearly independent over $A(R)$, then 

$\{\exp(a_1),...,\exp(a_n)\}\subset R'$ is algebraically  independent over $R$.
\item There is no proper partial E-subring of $R'$ satisfying these conditions.
\end{itemize}

\end{defn}

It is worth noting at this point that the fourth condition implies that $A(R')= R$. The next lemma easily follows from equivalent constructions in
\cite{Lou E-rings},\cite{Angus E-rings}.

\begin{lemma}\label{Unique Free}
Given any partial E-ring $R$, there is a free partial E-ring extension $R'$ of $R$.  Furthermore, if $R'$ and $S'$ are free partial E-ring extensions of $R$, then $R' \cong S'$.
\end{lemma}

\begin{proof}
 Let $F$ be a large algebraically closed field extension of $R$. Let
  
  $\class{b_i}{i\in I}$ be a $\Q$-basis of $R$ over $A(R)$, and for
  each $i \in I$ and $q \in \Q$, choose $d_{i,q} \in F$ such
  that $\class{d_{i,1}}{i\in I}$ is algebraically independent over $R$, and for all $s \in \Z$, $d_{i,q}^s =
  d_{i,qs}$.

  Extend $\exp_R$ to $\exp_{R'}$ by defining $\exp_{R'}(qb_i) =
  d_{i,q}$ for $q \in \Q$ and $i \in I$, and extending additively. Let
  $R'$ be the subring of $F$ generated by $R$ and all the
  $d_{i,q}$. Clearly $R'$ is a free partial E-ring extension of $R$

Let $S'$ be a different free partial E-ring extension of $R$. Consider $\widehat{d}_{i,q} \in S'$ where $\exp_{S'}(qb_i)=d_{i,q}$.  At this point note that the set $\{\widehat{d}_{i,1}: i\in I\}$ is algebraically independent over $R$ since $S'$ is a free partial E-ring extension of $R$. The subring of  $S'$ generated by $R$ and $\widehat{d}_{i,q}$ is also a partial E-ring extension of R.  By minimality of free partial E-ring extensions, $S'$ must be generated as a ring by $R$ and $\widehat{d}_{i,q}$. Our claim is that $S'$ is isomorphic to $R'$.  Consider the ring homomorphism $\phi : R' \to S'$ defined by $$\phi(qb_i) = qb_i \text{ and } \phi (d_{i,q}) = \widehat{d}_{i,q}.$$  Consider an algebraically closed field containing both $S'$ and $R'$.  Then, there is an automorphism of this algebraically closed field which fixes the algebraic closure of $R$, sends $d_{i,1}$ to $\widehat{d}_{i,1}$ and sends any coherent system of roots of $d_{i,1}$ to any coherent system of roots of $\widehat{d}_{i,1}$. Thus $\phi$ extends to an automorphism of this algebraically closed field. If we restrict this automorphism to $R'$, the image is $S'$.  It is easy to check that $\phi$ preserves the exponential map. Thus, $S'$ is isomorphic to $R'$ as partial E-ring.

\end{proof}

For any given partial E-ring, we use the prime notation to denote the free extension, i.e. if R is a partial E-ring, $R'$ is the free partial E-ring extension of R. We now connect free extensions to formally real fields via this next lemma.

\begin{lemma}
Suppose $R$ is a formally real partial E-ring.  Then, $R'$ is formally real.
\end{lemma}

\begin{proof}
Let $\{d_{i,q}: i \in I\}$ be as in the proof of Lemma ~\ref{Unique Free}.  Consider $R[\{d_{i,1}: i\in I\}]$, the ring extension of $R$ generated by $\{d_{i,1}: i\in I\}$.  This is a purely transcendental extension of $R$ and is thus formally real. If we extend an ordering on $R$ such that $d_{i,1}$ is positive for all $i\in I$, then any real closure of $R[\{d_{i,1}: i\in I\}]$ in a large algebraically closed field extension of $R[\{d_{i,1}: i\in I\}]$  will contain a consistent system of $n^{th}$ roots $\{d_{i,\frac{1}{n}}:i \in I, n \in \N\}$.  Thus, $R'$ is a subring of a real closed field and is thus formally real.
\end{proof}

\section{Countable Real Closed Exponential Fields}

The goal of this section is to prove Theorem 3.  Consider a chain of subrings of $\mK$ of the following form.

$$Q_0 \into Q_1 \into Q_2 \into \cdots$$

\noindent where $Q_0 \subseteq \Q^{alg}$ and $[Q_i \cup \exp(Q_i)] \subseteq Q_{i+1}  \subseteq [Q_i \cup \exp(Q_i)]^{alg}$. Let $\widetilde{Q} = \cup Q_i$.

\begin{defn}
Let $A \subseteq \widetilde{Q}$ be finite.  We say a set $D \subseteq \widetilde{Q}$ is an \emph{E-source of $A$} if for all $a\in A$,

\begin{enumerate}
\item $a \in (Q_0 \cup \exp(D))^{alg}$.
\item $\forall d\in D, d \in (Q_0 \cup \exp(D))^{alg}$.
\item $D$ is minimal such.
\end{enumerate}

\end{defn}

\noindent By the definition of $\widetilde{Q}$, E-sources always exist and are finite.  Furthermore, if $A \subseteq Q_i$ and $D$ is an E-source of $A$, then $D\subseteq Q_{i-1}$.

\begin{lemma}
E-sources are $\Q$-linearly independent.
\end{lemma}

\begin{proof}
Suppose $D$ is a decomposition of $A$ and $D$ is not $\Q$-linearly independent.  Then there is $d \in D$ such that $d \in \langle D \ \{d\} \rangle  $ and $\exp(d)$ is algebraic over $\exp(D - \{d\})$.  If $c\in \mK$ is such that $c \in [Q_0\cup \exp(D)]^{alg}$, then in fact, $c\in [Q_0 \cup \exp(D- \{d\})]^{alg}$.  Thus, $D-\{d\}$ contains an E-source of $A$ which contradicts minimality.

\end{proof}

\begin{lemma}
$[Q_i \cup \exp(Q_i)] \cong Q_i'$.
\end{lemma}

\begin{proof}

The statement of this lemma is \emph{a priori} puzzling, as it is not clear that $Q_i$ satisfies the domain condition of a partial E-subring of $\mK$.  However, in the following argument we prove that if $\{r_1,...,r_n\}\subset Q_i$ is $\Q$-linearly independent over $Q_{i-1}$, then $\{\exp(r_1,...,r_n)\}$ is algebraically independent over $Q_i$. This implies that the domain of $\exp_{Q_i}$ is exactly $Q_{i-1}$ and indeed, $Q_i$ is a partial E-ring extension of $Q_{i-1}$.  Since $[Q_{i-1} \cup \exp(Q_{i-1})]$ is the smallest subring of $Q_i$ to satisfy this, we have proven the lemma.

\bigskip

Let $\br \subset Q_i$ be $\Q$-linearly independent over $Q_{i-1}$.  Suppose $\exp(\br)$ is algebraically dependent over $Q_i$.
Then there is $\bs \subset Q_i$ such that $\exp(\br)$ is algebraically dependent over $\bs$ and $\br \subset \bs$.

Let $\bq \subset Q_0,\bt \in Q_{i-1}$ be such that $\{\bq,\bt\}$ is an E-source of $\bs$.  So each element of $\bt$ is algebraic over $\{\bq,\exp(\bt)\}$,
and each element of $\bs$ as well as each element of $\br$ is algebraic over $\{\bq,\exp(\bt)\}$.  Then $\exp(\br)$ is algebraically dependent over $\{\bq, \exp(\bt)\}$.  Thus $$\td(\bq,\bt,\br,\exp(\bq),\exp(\bt),\exp(\br))=$$ $$\td(\bq,\exp(\bq),\exp(bt)) + \td(\bt,\br,\exp(\br) / \bq,\exp(\bq),\exp(\bt))$$ $$\lneq |\bq| + |\bt| +|\br|.$$

Since $\mK$ satisfies Schanuel's conjecture, we conclude that $\{\bq,\bt,\br\}$ is $\Q$-linearly dependent.  Since $\{\bq,\bt\}$ is an E-source and thus $\Q$-linearly independent and a subset of $Q_{i-1}$, this implies that $\br$ is $\Q$-linearly dependent over $Q_{i-1}$.

\end{proof}

\begin{cor}
If $Q_i$ is formally real, then $[Q_i \cup \exp(Q_i)]$ is formally real.

\end{cor}

\begin{cor}
Consider the chain $$Q_0 \into Q_1 \into Q_2 \into \cdots$$ where $Q_0 = \Q^{rc}$ and $Q_{i+1}$ is a real closure of $[Q_i \cup \exp(Q_i)]$.  Then the union $\widetilde{Q}$ is a real closed exponential subfield of $\mK$.

\end{cor}

\smallskip

In order to define the real closure of a formally real ring $R$, the order must be fixed.  We have shown that if $R$ is formally real, then $R'$ will be formally real and an element in $R'$ transcendental of $R$ can satisfy any positive cut over $\Q$ which is transcendental over $R$.  Notice that these positive transcendental cuts are actually types over the empty set.  Thus ,isomorphic real closed exponential fields must satisfy the same cuts over $\Q$ and every positive transcendental cut is satisfiable in some construction of a real closed exponential $\widetilde{Q}$.  Since a given $\widetilde{Q}$ is countable and can only satisfy countably many types, there are uncountably many non-isomorphic constructions of a countably real closed exponential field $\widetilde{Q}$.

\bigskip

First notice that $\exp_\mK$ is injective when restricted to any of these countable real closed exponential fields.  To see this, consider first the exponential map restricted to $\Q^{rc}$.  Schanuel's conjecture implies that the kernel is trivial.  Now consider an element $q \in Q_i$ where $q \notin Q_{i-1}$.  We have shown that $\exp_\mK(q)$ is transcendental over $Q_i$.  Thus, the kernel of the exponential map restricted to $\widetilde{Q}$ is trivial. At each stage we have shown the extension to be free and then we took the real closure.  If at stage $n$, we require that $\{d_{i,1}\}$ from the construction of the free extension has the same order type over the image of $\exp_{Q_{n-1}}$ as $\class{b_i}{i\in I}$ has over $A(Q_{n-1})$, then the exponential map will be order preserving, and since $\{d_{i,1}\}$ are algebraically independent over $Q_{n-1}$, we can do this.  Then, if you notice that at least when constructing $Q_1$ any positive cut can be satisfied, there are still continuum many real closed exponential subfields of $\mK$ each with an order preserving exponential map.

\section{Adding Logs}

In this section, we will prove by induction that we can construct the following chain of partial exponential rings:

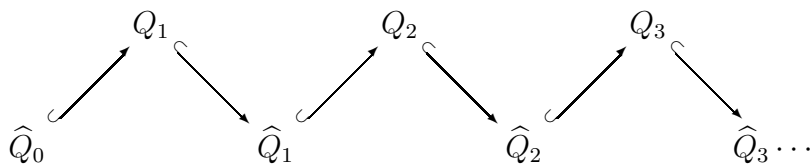
\begin{diagram}[height=2em,width=2em]
    &&Q_1&&&&Q_2&&&&Q_3\\
    &\ruInto>{} &&\rdInto>{} &&\ruInto>{} &&\rdInto>{} &&\ruInto>{} &&\rdInto>{} &&\ruInto>{} &&\rdInto>{}  \\
    \widehat{Q}_0&&&&\widehat{Q}_1&&&&\widehat{Q}_2&&&&\widehat{Q}_3\cdots
  \end{diagram}

where $\widehat{Q}_0 = \Q^{rc}$

$$Q_{i+1} =  [\widehat{Q}_i \cup \exp (\widehat{Q}_i)]^{rc} $$

and

$$\widehat{Q}_{i+1} = [Q_{i+1} \cup \log(Q_{i+1}^{>0})]^{rc}. $$

Let $\widetilde{Q} = \cup Q_i$.

Similarly to the proofs we did earlier in this paper, the proof that at each stage of this construction the rings we are consider are formally real will rely on showing that we are essentially dealing with purely transcendental extensions.  In order to understand the construction of $\widetilde{Q}$, it is useful to know what the expected domain and image of the exponential map are at each stage and to keep track of notation.  We will show that

\begin{itemize}
\item  $[\widehat{Q}_i \cup \exp (\widehat{Q}_i)] \cong \widehat{Q}_i[E_{i+1}^\Q]$, the free extension $\widehat{Q}_i$.  Here, we are denoting the algebraically independent set $\class{d_{i,1}}{i\in I}$ from the construction of the free extension $E_{i+1}$, and the set $\{d_{i,q}: i\in I,q\in \Q\}$ from this stage of the construction we denote $E_i^\Q$,
\item $[Q_{i+1} \cup \log(Q_{i+1}^{>0})]\cong Q_{i+1}[L_{i+1}]$ where $L_{i+1}$ is a set which is algebraically independent over $Q_{i+1}$,
\end{itemize}

and that the domain and image of the map are as small as possible at each stage, i.e.,

\begin{itemize}
\item $\dom(\exp_{\widehat{Q}_i})$ is the $\Q$ additively linear span of $\widehat{Q}_{i-1} \cup L_i$.
\item $\img(\exp_{Q_i})$ is the $\Q$ multiplicative span of $Q_{i-1}^{>0} \cup E_i$.

\end{itemize}

\begin{defn}
Let $\bs \subset \widehat{Q}_n$.  We say $\{E,L\}:= \{\be_1,...,\be_n,\bl_1,...,\bl_{n} : \be_i \subset E_i, \bl_i \subset L_i\}$  is an \emph{LE-source of $\bs$} if

\begin{itemize}
\item For all $s \in \bs$, $s$ is algebraic over $\{E,L\}$.
\item For all $e \in e_i$ for $i=1,...,n$, $e$ is algebraic over $\{E,L\}$.
\item For all $l \in l_i$ for $i=1,...,n-1$, $l$ is algebraic over $\{E,L\}$.
\item $\{E,L\}$ is minimal such.
\end{itemize}

If $\bs \subset Q_n$, then we use the same definition but note that $\{E,L\}:= \{\be_1,...,\be_n,\bl_1,...,\bl_{n-1} : \be_i \subset E_i, \bl_i \subset L_i\}$, since we have not yet added the logs at the $n^{th}$ stage. This will be key in the proofs below.

\end{defn}

Notice now that $Q_1$ and $T_1$ exist by the proofs done at the end of the previous section.  Thus, we have a base case for the induction and we will assume for purposes of induction that we have carried out the construction though $Q_n$ or $\widehat{Q}_n$ and chosen an ordering at each stage so that all elements of $E_i^\Q$ are positive. Also notice that for $l\in L_i$, $\exp(l) \in Q_i$.  Similarly, for $e \in E_i$, $\log(e) \in \widehat{Q}_{i-1}$. By the induction assumption that we have carried out the construction up to and including $Q_n$ or $\widehat{Q}_n$, LE-sources exist, are finite as defined, and minimality guarantees that they are algebraically independent as sets.  We will need the following claim about LE-sources:

\begin{claim}
Let $\{E,L\}:= \{\be_1,...,\be_n,\bl_1,...,\bl_{n} : \be_i \subset E_i, \bl_i \subset L_i\}$ be an LE-source for some finite subset of $\widehat{Q}_n$.
Let $\bq = \log(\be_1) \subset \Q^{rc}$.  Then, the set $\{\bq,\bl_1,\log(\be_2),\bl_2,...,\bl_{n-1}, \log(\be_n),\bl_n\}$ is $\Q$- linearly independent.

\end{claim}

\begin{proof}

Suppose $n=1$. Then, by induction, since $E_1$ is algebraically independent and $\be_1 \subset E_1$, $\log(\be_1)$ must be $\Q$-linearly independent.  Since $\bl_1$ is algebraically independent over $Q_1$, and thus $\Q$-linearly independent over $Q_1$, $\{\bq,\bl_1\}$ is $\Q$-linearly independent.  Similary, if $\{\bq,\bl_1,...,\log(\be_i)\}$ is linearly independent, then since this set is in $\widehat{Q}_{i-1}$ and $L_i$ is algebraically independent over $Q_i$ and thus over $\widehat{Q}_{i-1}$, the set $\{\bq,\bl_1,...,\log(\be_i),\bl_i\}$ is $\Q$-linearly independent.

Now suppose the set up to $\bl_i$ is $\Q$-linearly independent. Then, since $E_{i+1}$ is algebraically independent over $\widehat{Q}_i$ and $\be_{i+1}\subset E_{i+1}$, we know that $\log(\be_{i+1})\subset \widehat{Q}_i$ is $Q$-linearly independent over the domain of the exponential map in $\widehat{Q}_i$.  Since the domain contains $\widehat{Q}_{i-1} \cup L_i$ and the set up to $\bl_i$ is contained in $\widehat{Q}_{i-1} \cup L_i$, we have that the set up to $\log(\be_{i+1})$ is $\Q$-linearly independent.

\end{proof}

We are now ready to prove that the construction can be extended from $Q_n$ to $\widehat{Q}_n$.

\begin{lemma}
 Notice that $Q_n^{>0}$ is a $\Q$-multiplicatively linear space since $Q_n$ is real closed and every positive element has a unique positive $n^{th}$ root. Suppose $\ba \in Q_n^{>0}$ is $\Q$- multiplicatively independent over $Q_{n-1}^{>0} \cup E_n$.  Then $\log(\ba)$ is algebraically independent over $Q_n$.

\end{lemma}

\noindent This lemma will guarantee that $L_n$ exists as described and that $[Q_{n} \cup \log(Q_{n}^{>0})]\cong Q_{n}[L_{n}]$.

\begin{proof}
Suppose $\log(\ba)$ is algebraically dependent over $Q_n$.  Then there is $\bs \subset Q_n$ such that $\log(\ba)$ is algebraically dependent over $\bs$ and without loss of generality, we may assume each $a \in \ba$ is algebraic over $\bs$.  Let $\{E,L\}= \{\be_1,...,\be_n,\bl_1,...,\bl_{n-1} \} $ be an LE-source of $\bs$.  Consider $$\{\bq,\bl_1,\log(\be_2),\bl_2,...,\bl_{n-1}, \log(\be_n), \log(\ba), \be_1, \exp(\bl_1),..., \be_n, \ba\}$$ where $\bq = \log(\be_1) \subset \Q^{rc}$ and thus the second have of the set we are considering is the exponential image of the first half.  By definition of an LE-source, we compute

$$\td(\bq,\bl_1,\log(\be_2),\bl_2,...,\bl_{n-1}, \log(\be_n), \log(\ba), \be_1, \exp(\bl_1),..., \be_n, \ba ) < |\{E,L\}|+ |\ba|.$$

So, by Schanuel's conjecture, we have that 

$\{\bq,\bl_1,\log(\be_2),\bl_2,...,\bl_{n-1}, \log(\be_n), \log(\ba)\}$ is $\Q$-linearly dependent. By the claim, we know that

$\{\bq,\bl_1,\log(\be_2),\bl_2,...,\bl_{n-1},\log(\be_n)\}$ is $\Q$-linearly independent.  Thus, $\log(\ba)$ is $\Q$-linearly dependent over

$\{\bq,\bl_1,\log(\be_2),\bl_2,...,\bl_{n-1},\log(\be_n)\}$.  So $\ba$ is $\Q$- multiplicatively dependent over

$\{\be_1,\exp(\bl_1),,....,\exp(\bl_{n-1}),\be_n\} \subset Q_{n-1}^{>0} \cup E_n$.

\end{proof}

Thus, if $Q_n$ is formally real, then so is the purely transcendental extension $Q_{n}[L_{n}]$ and we can take the real closure as $\widehat{Q}_n$.

\bigskip

\bigskip

\noindent As in the previous section, the following lemma will guarantee that $[\widehat{Q}_n \cup \exp (\widehat{Q}_n)]$ is indeed the free extension of $\widehat{Q}_n$ and that the domain of the exponential map is precisely what we described above.

\begin{lemma}
Suppose $\ba \subset \widehat{Q}_n$ is $\Q$-linearly independent over $\widehat{Q}_{n-1} \cup L_n$.  Then $\exp(\ba)$ is algebraically independent over $\widehat{Q}_n$.

\end{lemma}

\begin{proof}

Suppose $\exp(\ba)$ is algebraically dependent over $\widehat{Q}_n$. Then there is $\bs \subset \widehat{Q}_n$ such that $\exp(\ba)$ is algebraically independent over $\bs$ and we may assume without loss of generality that each $a \in \ba$ is algebraic over $\bs$.  Let $\{E,L\}= \{\be_1,...,\be_n,\bl_1,...,\bl_{n}\}$ be an LE-source of $\bs$.  Now, where $\bq = \log(\bt_1) \subset \Q^{rc}$, we have $$\td(\bq,\bl_1,\log(\be_2),\bl_2,...,\log(\be_n),\bl_n,\ba,\be_1,\exp(\bl_1),....,\be_n,\exp(\bl_n),\exp(\ba))$$
$$|\{E,L\}| + |\ba|.$$

So, by Schanuel's conjecture, $\ba$ is $\Q$-linearly dependent over 

$\{\bq,\bl_1,\log(\be_2),\bl_2,...,\log(\be_n),\bl_n\} \subset \widehat{Q}_{n-1} \cup L_n$.

\end{proof}

Thus, if $\widehat{Q}_n$ is formally real, then so is the free extension $[\widehat{Q}_n \cup \exp (\widehat{Q}_n)]$ and we can take the real closure to get $Q_{n+1}$.  This completes the proof that the chain exists as described and that at each stage the domain and image of the exponential map are precisely the minimal possible set.

\bigskip

\bigskip

\noindent To finish the proof of the theorem, notice that at each stage we are adding transcendental elements.  If we make the $L_i$ satisfy the same order type over the previous domain as their exponential image satisfies over the previous image, and make the $E_i$ satisfy the same order type over the previous image as their preimage satisfies over the previous domain, the exponential map will be order preserving.  As there are clearly continuum many positive cuts that can be satisfied when constructing $Q_1$ and only countably many are satisfied in any one construction, we have proven the theorem.

\end{document}

%% file: macros.tex
\DeclareMathOperator{\td}{td}  

\DeclareMathOperator{\dom}{dom}   

\newcommand{\N}{\ensuremath{\mathbb{N}}}
\newcommand{\Z}{\ensuremath{\mathbb{Z}}}
\newcommand{\Q}{\ensuremath{\mathbb{Q}}}
\newcommand{\R}{\ensuremath{\mathbb{R}}}
\newcommand{\C}{\ensuremath{\mathbb{C}}}

\renewcommand{\phi}{\varphi}

\newcommand{\class}[2]{\ensuremath{\left\{ #1 \,\left|\, #2 \right.\right\}}}

\newcommand{\into}{\hookrightarrow}

\newcommand{\nstrong}{\ensuremath{\not\kern-4pt\lhd\;}} 

\newbox\noforkbox \newdimen\forklinewidth
\forklinewidth=0.3pt \setbox0\hbox{$\textstyle\smile$}
\setbox1\hbox to \wd0{\hfil\vrule width \forklinewidth depth-2pt
  height 10pt \hfil}
\wd1=0 cm \setbox\noforkbox\hbox{\lower 2pt\box1\lower
2pt\box0\relax}
\def\unionstick{\mathop{\copy\noforkbox}\limits}

\def\nonfork_#1{\unionstick_{\textstyle #1}}

\setbox0\hbox{$\textstyle\smile$} \setbox1\hbox to \wd0{\hfil{\sl
/\/}\hfil} \setbox2\hbox to \wd0{\hfil\vrule height 10pt depth
-2pt width
                \forklinewidth\hfil}
\newbox\doesforkbox
\setbox\doesforkbox\hbox{\lower 2pt\box1 \lower
2pt\box2\lower2pt\box0\relax}
\def\nunionstick{\mathop{\copy\doesforkbox}\limits}

\def\fork_#1{\nunionstick_{\textstyle #1}}

